\documentclass[a4paper
               ,11pt
               ,DIV=10
               ,oneside
               ,aps
               ,notitlepage
               ,nofootinbib
               ,superscriptaddress
               ,tightenlines
               ]
               {scrartcl}

\usepackage[english]{babel}
\usepackage{amsmath}
\usepackage{hyperref}
\usepackage[numbers]{natbib}
\usepackage{amsthm}
\usepackage{dsfont}
\usepackage[utf8]{inputenc}
\usepackage[plain]{fancyref} %references with object type, no pages
\usepackage[dvipsnames]{xcolor}
\usepackage{physics}
\usepackage{braket}
\usepackage{amsfonts}
\usepackage{authblk}
\usepackage{setspace}
\usepackage{hyphenat}
\usepackage{breakurl}
\newtheorem{thm}{Theorem}[section]
\newtheorem{lem}[thm]{Lemma}

\newtheorem{defi}[thm]{Definition}
\newtheorem{prop}[thm]{Proposition}

\renewcommand{\AA}{\mathcal{A}}

\newcommand{\MM}{\mathcal{M}}

\newcommand{\OO}{\mathcal{O}}

\newcommand{\setN}{\mathds{N}}
\newcommand{\setP}{\mathds{P}}
\newcommand{\setR}{\mathds{R}}

\newcommand{\idop}{\mathds{1}} %Identity Operator
\renewcommand{\tr}[1]{\mathrm{Tr}\left(#1\right)}
\newcommand{\sym}{\mathrm{sym}} %Subscript for symmetric matrices

\newcommand{\M}{\mathcal{M}}
\newcommand{\mscalar}[2]{\langle #1, #2 \rangle}
\newcommand*{\fancyrefthmlabelprefix}{thm}
\newcommand*{\fancyreflemlabelprefix}{lem}
\newcommand*{\fancyrefcorlabelprefix}{cor}
\newcommand*{\fancyrefdefilabelprefix}{defi}
\frefformat{plain}{\fancyreflemlabelprefix}{lemma\fancyrefdefaultspacing#1}
\Frefformat{plain}{\fancyreflemlabelprefix}{Lemma\fancyrefdefaultspacing#1}
\frefformat{plain}{\fancyrefthmlabelprefix}{theorem\fancyrefdefaultspacing#1}
\Frefformat{plain}{\fancyrefthmlabelprefix}{Theorem\fancyrefdefaultspacing#1}
\frefformat{plain}{\fancyrefcorlabelprefix}{corollary\fancyrefdefaultspacing#1}
\Frefformat{plain}{\fancyrefcorlabelprefix}{Corollary\fancyrefdefaultspacing#1}
\frefformat{plain}{\fancyrefdefilabelprefix}{definition\fancyrefdefaultspacing#1}
\Frefformat{plain}{\fancyrefdefilabelprefix}{Definition\fancyrefdefaultspacing#1}

\begin{document}

\title{Dimensionality reduction of SDPs through sketching}

\author[1]{Andreas Bluhm\thanks{andreas.bluhm@ma.tum.de}~}
\author[1]{Daniel Stilck Fran\c{c}a\thanks{dsfranca@mytum.de}~}
\affil[1]{\small{\emph{Department of Mathematics, Technical University of Munich, 85748 Garching, Germany}}}
\maketitle

\begin{abstract}
We show how to sketch semidefinite programs (SDPs) using positive maps in order to reduce their dimension. 
More precisely, we use Johnson\hyp{}Lindenstrauss transforms to produce a smaller SDP whose solution preserves feasibility or approximates the value of the original problem with high probability. 
These techniques allow to improve both complexity and storage space requirements. They apply to problems in which the Schatten 1-norm of the matrices specifying the SDP and also of a solution to the problem is 
constant in the problem size. Furthermore, we provide some results which clarify the limitations of positive, linear sketches in this setting. 
%Finally, we discuss numerical examples to benchmark our methods.
\\
\\
\noindent \textbf{Keywords:} semidefinite programmming; sketching; dimensionality reduction; Johnson-Lindenstrauss transforms.
\\
\noindent \textbf{MSC Subjects:} \emph{Main:} 90C22 \emph{Secondary:} 15A39, 15B48, 68W20

\end{abstract}
\tableofcontents

\begin{section}{Introduction}

Semidefinite programs (SDPs) are a prominent class of optimization problems \cite{lasserre2016handbook}. They
have applications across different areas of science and mathematics, such as discrete optimization \cite{SDPdiscopt} or control theory \cite{SDPcontrol}.

However, although there are many different algorithms that solve an SDP up to an error $\epsilon$ in
a time that scales polynomially with the dimension and logarithmically with $\epsilon^{-1}$ \cite{SDPcomplexity},
solving large instances of SDPs still remains a challenge. 
This is not only due to the fact that the number and cost of
the iterations scale superquadratically with the dimension for most algorithms to solve SDPs, but also due to the fact that the memory required
 to solve large instances is beyond current capabilities. 
This has therefore motivated research on algorithms that can solve SDPs, or at least obtain an approximate solution, with less memory requirements. One example are the so called first order methods, which were developed to remedy the high memory requirements of interior point methods; see \cite{Renegar2014} and references therein.
Another such example is the recent \cite{sketchydecisions}, where ideas similar to ours were
applied to achieve optimal storage requirements necessary to solve a certain class of SDPs. While the latter work proposes a new way to solve an SDP using linear sketches,
our approach relies on standard convex optimization methods.

In this work, we develop algorithms to estimate the value of an SDP with linear inequality constraints
and to determine if a given linear matrix inequality (LMI) is feasible or not.
These algorithms convert the original problem to one of the same type, but of smaller dimension, which we call the sketched problem.
Subsequently, this new problem can be solved with the same techniques as the original one, but potentially using less memory
and achieving a smaller runtime. Therefore, we call this a black box algorithm. With high probability an optimal solution to the sketched problem
allows us to obtain a good approximation of the value or to test the feasibility of the problem.

In the case of LMIs, if the  sketched problem is infeasible, we obtain a certificate that the original problem is also infeasible.
If the sketched problem is feasible, we are able to infer that the original problem is either ``close to feasible''
or feasible with high  probability, under some technical assumptions. 

In the case of estimating the value of SDPs, we are able to give an upper bound that holds with high probability
and a lower bound on the value of the SDP from the value of the sketched problem, again under some technical assumptions.
For a certain class of SDPs, which includes the so-called semidefinite packing problems \cite{Iyengar2005},
we are able to find a feasible point of the original problem which is close to the optimal point and most technical aspects simplify significantly. 

Our algorithms work by conjugating the matrices that define the constraints of the SDP with Johnson-Lindenstrauss transforms \cite{Woodruff2014},
thereby preserving the structure of the problem. Similar ideas have been proposed to reduce the memory usage and complexity of solving linear programs \cite{2015arXiv150700990V}.
While those techniques aim to reduce the number of constraints, our goal is to reduce the dimension of the matrices involved.

Unfortunately, the dimension of the sketch needed to have a fixed error with high probability scales with the Schatten $1$-norm of both the
constraints and of an optimal solution to the SDP, which significantly restricts the class of problems to which these methods can be applied.
We are able to show that one cannot significantly improve this scaling and that one cannot sketch general SDPs using linear maps. 

This paper is organized as follows:
in \Fref{sec:preliminaries}, we fix our notation and recall some basic notions from matrix analysis, Johnson-Lindenstrauss transforms,
semidefinite programs and convex analysis which we will need throughout the paper.
We then proceed to show how to sketch the Hilbert-Schmidt scalar product with positive maps in \Fref{sec:sketchhs}.
We apply these techniques in \Fref{sec:feasibility} to show how to certify that certain LMIs are infeasible by
showing the infeasibility of an LMI of smaller dimension. 
In \Fref{sec:approxvalue}, we apply similar ideas to estimate the value of an SDP with linear inequality constraints by solving an SDP of lower dimension.
%This is followed by
We conclude with a discussion of the possible gains in the complexity of solving these problems and for the memory requirements in \Fref{sec:complexity}. 
 
\end{section}

\begin{section}{Preliminaries}\label{sec:preliminaries}
We begin by fixing our notation. For brevity, we will write the set $\{1,\ldots,d\}$ as $[d]$. The positive vectors will denoted by $\setR^m_+:=\set{x\in\setR^m:x_i\geq0}$.
The set of real $d\times D$ matrices
will be written as $\M_{d,D}$ and just $\M_d$ if $d=D$. 
We will denote by $\M_d^{sym}$ the set of symmetric $d\times d$ matrices.
For $A\in\M_d$, $A^T$ will denote the transpose of $A$.
To avoid cumbersome notation and redundant theorems, we will the statements only for real matrices. However, note that all statements
translate to the complex case in a straightforward fashion. 
For $A\in\M_d^{sym}$ we will write $A\geq0$ if $A$ is positive semidefinite. We will denote the cone of $d\times d$ positive semidefinite matrices
by $\mathcal{S}_d^+$ and its interior, the positive definite matrices, by $\mathcal{S}_d^{++}$. For the Schatten $p$-norm for $p\in[1,\infty)$ of a matrix $A\in\M_d$
we will write
\begin{equation*}
\|A\|_p := \Tr[(A^T A)^{\frac{p}{2}}]^{\frac{1}{p}}.
\end{equation*} 
The $p=\infty$ norm is the usual operator norm. The Schatten-2 norm is often called the Hilbert-Schmidt (HS) norm and is induced by the Hilbert-Schmidt scalar product, which is given by
$\langle A,B\rangle_{HS}=\tr{A^TB}$. 

A linear map $\Phi:\M_D\to\M_d$ is called positive if $\Phi(\mathcal{S}_D^+)\subseteq\mathcal{S}_d^+$.
We will mostly consider maps of the form $\Phi(X)=SXS^T$ with $S\in\M_{d,D}$.

The following families of matrices will play a crucial role for our purposes:
\begin{defi}[Johnson-Lindenstrauss transform]
A random matrix $S \in \M_{d,D}$ is a Johnson-Lindenstrauss transform (JLT) with parameters $(\epsilon,\delta,k)$ 
if with probability at least $1 - \delta$, for any $k$-element subset $V \subseteq \mathds{R}^D$, for all $v, w \in V$ it holds that
\begin{equation*}
|\mscalar{Sv}{Sw} - \mscalar{v}{w}| \leq \epsilon \norm{v}_2 \norm{w}_2.
\end{equation*} 
\end{defi}
Note that one usually only demands that the norm of the vectors involved is distorted by at most $\epsilon$ in the definition
of JLTs, but this is equivalent to the definition we chose by the polarization identity.
There are many different examples  of JLTs in the literature and we refer to \cite{Woodruff2014} and references therein for more details.
Most of the constructions of JLTs focus on real matrices, but the generalization to complex matrices is straightforward. % in \Fref{sec:complexstuff} we show how to lift some of these results to cover complex matrices.
One simple example are random matrices $S=\frac{1}{\sqrt{d}}R\in\M_{d,D}$ with $R$ having i.i.d. standard
Gaussian random variables, which can be shown to be $(\epsilon,\delta,k)$-JLT if $d=\Omega(\epsilon^{-2}\log(k\delta^{-1}))$~\cite[Lemma 2.12]{Woodruff2014}.

It will later be of advantage to our algorithm to consider JLTs with a desired sparsity $s$ and we mention the following almost optimal result. We refer to \cite[Section 1.1]{sparsejltnelson} for a proof and remark that the proof is constructive.

\begin{thm}[Sparse JLT {\cite[Section 1.1]{sparsejltnelson}}] \label{thm:bestJLT}
There is an $(\epsilon,\delta,k)$-JLT $S\in\M_{d,D}$ with $d=\mathcal{O}\left(\epsilon^{-2}\log(k\delta^{-1})\right)$
and $s=\mathcal{O}(\epsilon^{-1}\log(k\delta^{-1}))$ nonzero entries per column.
\end{thm}
Given some JLT $S\in\M_{d,D}$, the positive map $\Phi:\M_D\to\M_d$, $X\mapsto SXS^T$ will be called
the sketching map and $d$ the sketching dimension.

We will now fix our notation for semidefinite programs.
Semidefinite programs are a class
of optimization problems in which a linear functional
is optimized under linear constraints over the set of positive semidefinite matrices. We refer to \cite{lasserre2016handbook} for an introduction
to the topic. There are many equivalent ways of formulating SDPs. In this work, we will assume w.l.o.g. that the SDPs are given in 
the following form:

\begin{defi}[Sketchable SDP] \label{defi:sketchablesdp}
Let $A, B_1, \ldots, B_m \in \MM_{D}^{\sym}$ and $\gamma_1,\ldots,\gamma_m\in\setR$. We will call the constrained optimization 
problem 
\begin{align}\label{sketchablesdp} 
\mathrm{maximize\qquad} & \tr{A X} \nonumber\\
\mathrm{subject~to\qquad} & \tr{B_i X} \leq \gamma_i, \qquad i \in [m]\\
					& X \geq 0 \nonumber,
\end{align}
a sketchable SDP.
\end{defi}
Sometimes we will also refer to a sketchable SDP as the original problem.
We will see later how to approximate the value of these SDPs.
SDPs have a rich duality theory~\cite{lasserre2016handbook}. The dual problem of a sketchable SDP is given by the following:
\begin{align}\label{dualsketchablesdp}
\mathrm{minimize\qquad} & \langle c,\gamma\rangle \nonumber\\
\mathrm{subject~to\qquad} & \sum_{i=1}^m c_iB_i-A\geq 0\\
& c \in \setR^m_+ \nonumber,
\end{align}
where $\gamma\in\setR^m$ is the vector with coefficients $\gamma_i$. 
SDPs and LMIs will be called feasible if there is at least one point satisfying all the constraints, otherwise we will call them
infeasible. A sketchable SDP will be called strictly feasible if there is a point $X > 0$ such that all the constraints
in \eqref{sketchablesdp} are satisfied with strict inequality.
Under some conditions, such as Slater's condition \cite{lasserre2016handbook}, the primal problem \eqref{sketchablesdp} and the dual problem \eqref{dualsketchablesdp} have the same value. This is called strong duality.

We will need some standard concepts from convex analysis. Given $a_1,\ldots,a_n\in V$ for a vector space $V$, we denote by $\text{conv}\{a_1,\ldots,a_n\}$ the convex hull of the points. 
By $\text{cone}\{a_1,\ldots,a_n\}$ we will denote
the cone generated by these elements and a convex cone $C$ will be called pointed if $C\cap-C=\{0\}$.

\end{section}

\section{Sketching the Hilbert-Schmidt product with positive maps}\label{sec:sketchhs}

One of our main ingredients to sketch an SDP or LMI will be a random positive map $\Phi:\M_D\to\M_d$
that preserves the Hilbert-Schmidt scalar product with high probability. We demand positivity to assure that
the structure of the SDP or LMI is preserved.
Below, we first consider the example $\Phi(X)=SXS^T$ with $S$ a JLT. A similar estimate was proved in \cite{Stark2016}
 for a different application.  
\begin{lem}\label{lem:sketchhs}
Let $B_1,\ldots,B_m\in \M_D^\sym$ and $S\in\M_{d,D}$ be an $(\epsilon,\delta,k)$-JLT with $\epsilon\leq1$ and $k$ such that
\[
k\geq\sum\limits_{i=1}^{m}\rank{B_i}. 
\]
Then
\begin{align}\label{equ:boundhs}
\setP\left[\forall i,j\in[m]:|\tr{SB_iS^TSB_jS^T}-\tr{B_iB_j}|\leq3\epsilon\norm{B_i}_1\norm{B_j}_1\right]\geq1-\delta. 
\end{align}

\end{lem}
\begin{proof}
Observe that the eigenvectors of the $B_i$ corresponding to nonzero eigenvalues of the $B_i$ form a subset of cardinality at most $k$ of $\mathds{R}^D$.
Let $A,B\in\{B_1,\ldots,B_m\}$.
As $S$ is an $(\epsilon,\delta,k)$-JLT, with probability at least $1-\delta$ we have for all normalized eigenvectors $a_i$ of 
$A$ and $b_j$ of $B$ that 
\begin{equation}\label{eq:boundjltscalar}
\big\lvert|\mscalar{Sa_i}{Sb_j}| - |\mscalar{a_i}{b_j}|\big\rvert \leq \epsilon
\end{equation}
by the reverse triangle inequality. We also have that for any $a_i$, $b_j$ 
\begin{equation*}
\norm{Sa_i}_2\leq\sqrt{1+\epsilon}, 
\quad  \norm{Sb_j}_2\leq\sqrt{1+\epsilon}, 
\end{equation*}
again by the fact that $S$ is a JLT.
As $\epsilon\leq1$ and by the Cauchy-Schwarz inequality, it follows that
\begin{equation}\label{eq:boundsumscalar}
|\mscalar{Sa_i}{Sb_j}| + |\mscalar{a_i}{b_j}|\leq3
\end{equation}
and hence, by multiplying \eqref{eq:boundsumscalar} with \eqref{eq:boundjltscalar},
\begin{equation}\label{eq:boundscalarsquared}
\big\lvert|\mscalar{Sa_i}{Sb_j}|^2-|\mscalar{a_i}{b_j}|^2\big\rvert\leq3\epsilon. 
\end{equation}
Now let $\lambda_i$ and $\mu_j$ be the eigenvalues of $A$ and
$B$, respectively. We have:
\begin{align*}
\big\lvert\tr{SAS^TSBS^T}-\tr{AB}\big\rvert&=\left|\sum_{i,j=1}^D\lambda_i\mu_j(|\mscalar{Sa_i}{Sb_j}|^2-|\mscalar{a_i}{b_j}|^2)\right|\\
&\leq3\epsilon\sum_{i,j=1}^D|\lambda_i||\mu_j|=3\epsilon\norm{A}_1\norm{B}_1
\end{align*}
with probability at least $1-\delta$.
As $A,B$ were arbitrary, the claim follows.
\end{proof}

The scaling of the error with the Schatten $1$-norm of the matrices involved in Lemma \ref{lem:sketchhs} is
highly undesirable, as the norm might grow linearly with the dimension. Applying JLTs for the Hilbert space $\M_D^{\text{sym}}$ would give a scaling of the error with the Schatten $2$-norm, but it would not necessarily preserve positivity of the matrices. 
The next theorem
shows that a scaling of the error with the Schatten $2$-norm of the matrices involved is not possible with positive maps if we want to achieve a non-trivial compression. Therefore, we cannot hope for a much better error dependence even with more advanced tools than the crude estimates which we have used.

\begin{thm}\label{thm:nogofrobeniussketch}
Let $\Phi:\M_D\to\M_d$ be a random positive map such that with strictly positive probability for any
$Y_1,\ldots\,Y_{D+1}\in\M_D$ and $0<\epsilon<\frac{1}{4}$ we have
\begin{equation}\label{eq:boundhs2}
|\tr{\Phi(Y_i)^T\Phi(Y_j)}-\tr{Y_i^T Y_j}|\leq \epsilon\|Y_i\|_2\|Y_j\|_2.
\end{equation}
Then $d=\Omega(D)$.
\end{thm}

\begin{proof}
Let $\{e_i\}_{1\leq i\leq D}$ be an orthonormal basis of $\mathds{R}^D$ and define $X_i=e_ie_i^T$.
As \Fref{eq:boundhs2} is satisfied with positive probability, there must exist a positive map
$\Phi:\M_D\to\M_d$ such that \Fref{eq:boundhs2} is satisfied for $Y_i=X_i$, $i\in[D]$, and 
$Y_{D+1}=\idop$.
 As the $X_i$ are orthonormal w.r.t.\ the Hilbert Schmidt scalar product and $\Phi$ is positive we have for $i,j\in[D]$
\begin{equation}\label{eq:scalarxi}
\tr{\Phi(X_i)\Phi(X_j)}\in
\begin{cases}
[0,\epsilon],\quad\text{for }i\not=j\\
[1-\epsilon,1+\epsilon], \quad\text{for }i=j.
\end{cases}
\end{equation}
Define the matrix $A\in\M_{D}$ with $(A)_{ij}=\tr{\Phi(X_i)\Phi(X_j)}$ for $i,j\in[D]$.
It is clear that $A$ is symmetric and that its entries are positive.
We have
\begin{equation*}
 \sum\limits_{i,j\in[D]}A_{ij}=\tr{\Phi(\idop)\Phi(\idop)}\in
\left[(1-\epsilon)D,(1+\epsilon)D\right].
\end{equation*}
As $A_{ii}\geq(1-\epsilon)$, it follows that 
\begin{equation}\label{eq:sumall}
\sum\limits_{i\not=j}A_{ij}\leq2\epsilon D. 
\end{equation}
Let 
\begin{equation*}
 J=\set{(i,j)\in[D]\times[D]|i\not=j,A_{ij}\leq \frac{1}{D}}.
\end{equation*}
It follows from \Fref{eq:sumall} that $|\Set{(i,j) \in [D] \times [D]| i \neq j, (i,j) \notin J}|\leq 2D^2\epsilon$ and so 
\begin{equation*}%\label{equ:sizej}
|J|\geq \left((1-2\epsilon)D^2-D\right).
\end{equation*}
Since for $(i, j)\in J$ also $(j,i) \in J$, we can write $J = (I \times I)\backslash\{(i,i)|i\in I\}$ for $I \subseteq [D]$. Thus,
\begin{equation*}
|J|=|I|(|I|-1)\geq((1-2\epsilon)D^2-D)\geq \left(\frac{1}{2}-2\epsilon\right)D^2.
\end{equation*}
for $D\geq2$. From this it follows that
\begin{equation*}
|I|^2\geq|I|(|I|-1)\geq\left(\frac{1}{2}-2\epsilon\right)D^2, 
\end{equation*}
and we finally obtain 
\begin{equation}\label{eq:lowerboundi}
|I|\geq \sqrt{1/2-2\epsilon}D. 
\end{equation}

Notice that it follows from \Fref{eq:scalarxi} that we may rescale all the $X_i$ to $X'_i$
such that $\tr{\Phi(X'_i)^2}=1$ and the pairwise scalar product still satisfies $\tr{\Phi(X'_i)\Phi(X'_j)}\leq\frac{1}{D(1-\epsilon)}$ for $(i,j)\in J$.
If there is an $N \in \setN$ such that $d>\sqrt{1/2-2\epsilon}D$ for all $D \geq N$, the claim follows. 
We therefore now suppose that $d\leq \sqrt{1/2-2\epsilon}D $. Hence, $d \leq |I|$ by \Fref{eq:lowerboundi}.
By the positivity of $\Phi$ and the fact that the $X'_i$ are positive semidefinite, we have that 
$\Phi(X_i')$ is positive semidefinite.
In \cite[Proposition 2.7]{Wolf2012} it is shown that for any set $\{P_i\}_{i \in I}$ of $|I|\geq d$ positive semidefinite matrices
in $\M_d$ such that $\tr{P_i^2}=1$ we have that
\begin{equation*}%\label{equ:boundpovm}
 \sum_{i\not=j}\tr{P_iP_j}^2\geq \frac{(|I|-d)^2|I|}{(|I|-1)d^2}.
\end{equation*}
By the definition of the set $J$, we have that
\begin{equation*}
\sum_{(i,j)\in J}\tr{\Phi(X_i')\Phi(X_j')}^2\leq \frac{|J|}{(1-\epsilon)^2D^2}\leq\frac{1}{(1-\epsilon)^2}, 
\end{equation*}
as $|J|\leq D^2$.
From \Fref{eq:lowerboundi} it follows that
\begin{equation*}
\frac{1}{(1-\epsilon)^2}\geq\left(\frac{\sqrt{1/2-2\epsilon}D}{d}-1\right)^2 
\end{equation*}
and after some elementary computations we finally obtain
\begin{equation*}
d\geq\frac{(1-\epsilon)\sqrt{1/2-2\epsilon}}{2-\epsilon}D. 
\end{equation*}
\end{proof}

It remains open if one could achieve a better compression for a sublinear number of matrices.
We also note that other theorems that restrict the possibility of dimensionality reduction using positive maps were proved in~\cite{harrowdimred},
although their results are restricted to maps that are in addition trace preserving and they also  demand that the distribution of maps is highly symmetric.

\section{Sketching linear matrix inequality feasibility problems}\label{sec:feasibility}
In this section we will show how to use JLTs to certify that certain
linear matrix inequalities (LMI) are infeasible by showing that an LMI of smaller dimension is infeasible. 
 The following lemma is similar in spirit to the well-known Farkas' lemma.
\begin{lem}\label{lem:strictlysephyper}
Let $A,B_1,\ldots,B_m\in\M_D^{\text{sym}}\backslash\set{0}$ such that 
\begin{equation}\label{ineq:matrixineqinfeasible1}
\sum_{i=1}^m c_iB_i-A\not\geq0   
\end{equation}
for all $c\in\setR^m_+$. Suppose further that 
\begin{equation*}
\Lambda=\mathrm{cone}\{B_1,\ldots,B_m\}
\end{equation*}
is pointed and $\Lambda\cap S_D^+=\{0\}$. Then there exists a  $\rho\in\mathcal{S}_D^+$ such that for all $i\in[m]$
\begin{align}\label{eq:proprho}
\tr{\rho B_i}<0,\quad\tr{-A\rho}<0\quad\text{and}\quad\tr{\rho}=1.
\end{align}

\end{lem}
\begin{proof}
Let $E=\text{conv}\{-A,B_1,\ldots,B_m\}$. We will show that $S^+_D\cap E=\emptyset$.
Suppose there exists an $X=-p_0A+\sum\limits_{i=1}^mp_iB_i\in S^+_D\cap E$ with $p\in[0,1]^{m+1}$. 
If $p_0>0$, we could rescale $X$ by $p_0^{-1}$ and obtain a feasible point for \eqref{ineq:matrixineqinfeasible1}, a contradiction.
If $p_0=0$ and $X\not=0$, this would in turn contradict $\Lambda\cap S_D^+=\{0\}$. And if $X=0$, the cone $\Lambda$ would not be pointed. 
From these arguments it follows that $0\not\in E$.
The set $E$ is therefore closed, convex, compact and disjoint from the convex and closed set $\mathcal{S}^+_D$. We may thus 
find a hyperplane that strictly separates $\mathcal{S}^+_D$ from $E$. That is, a $\rho\in\M_D^{\text{sym}}$ such that 
w.l.o.g. $\tr{\rho X}\geq0$ for all $X\in\mathcal{S}_D^+$, as $0\in\mathcal{S}^+_D$, and $\tr{Y\rho}<0$ for all $Y\in E$.
As $\tr{\rho X}\geq0$ for all $X\geq0$,
%X\not=0$ by the strict separation, 
it follows that $\rho$ is positive semidefinite and it is clear that by normalizing
$\rho$ we may choose $\rho$ with $\tr{\rho}=1$.
\end{proof}

The main idea is now to show that under these conditions we may sketch the hyperplane in a way that it still separates the set of positive
semidefinite matrices and the sketched version of the set $\{\sum_{i=1}^m\gamma_iB_i-A|\gamma_i\geq0 \}$.

\begin{thm}\label{thm:conversefeasible}
Let $A,B_1,\ldots,B_m\in\M_D^{\text{sym}}\backslash\set{0}$ such that they satisfy the assumptions of \Fref{lem:strictlysephyper}. Moreover, let $\rho\in\mathcal{S}_D^+$ be as in Equation \eqref{eq:proprho}.
Set
\begin{align*}
\epsilon=\frac{1}{6}\min\Set{\left|\frac{\tr{\rho B_1}}{\norm{B_1}_1}\right|,\ldots,\left|\frac{\tr{\rho B_m}}{\norm{B_m}_1}\right|,\left|\frac{\tr{\rho A}}{\norm{A}_1}\right|}
\end{align*}
and take $S\in\M_{d,D}$ to be an $(\epsilon,\delta,k)$-JLT. Here, 
\begin{equation*}
k \geq \rank{A} + \rank{\rho} + \sum_{i = 1}^m \rank{B_i}.
\end{equation*}
Then
\begin{align}\label{ineq:lmisketch}
\sum_{i=1}^mc_iSB_iS^T-SAS^T\not\geq0   
\end{align}
for all $c\in\setR^m_+$, with probability at least $1-\delta$.
\end{thm}
\begin{proof}
It should first be noted that $\rho$ exists and $|\tr{A\rho}|> 0$, $|\tr{B_i\rho}| >0$ for all $i \in [m]$ by Lemma \ref{lem:strictlysephyper}. Therefore, also $\epsilon > 0$.
The matrix $\rho$ defines a hyperplane that strictly separates the set 
\begin{align*}
E=\Set{\sum\limits_{i=1}^mc_iB_i-A|c\in\setR^{m}_+} 
\end{align*}
and $S_D^+$. We will now show that $S\rho S^T$ strictly separates the sets 
\begin{align*}
E_S=\Set{\sum\limits_{i=1}^mc_iSB_iS^T-SAS^T|c\in\setR^{m}_+} 
\end{align*}
and $S_d^+$ with probability at least $1-\delta$, from which the claim follows.
Note that by our choice of $\rho$ and $\epsilon$, it follows from \Fref{lem:sketchhs} that we have
\begin{align*}
\tr{S\rho S^T S B_i S^T}\leq\tr{\rho B_i}+3\epsilon\norm{B_i}_1<0 
\end{align*}
with probability at least $1-\delta$ and similarly for $-A$ instead of $B_i$. Therefore, it follows that $\mathrm{Tr}(Z S\rho S^T)<0$ for all $Z\in E_S$.
As $S\rho S^T$ is a positive semidefinite matrix, it follows that $\tr{Y S\rho S^T}\geq0$ for all
$Y\in\mathcal{S}_d^+$. We have therefore found a strictly separating hyperplane for $E_S$ and $S_D^+$ and the LMI \eqref{ineq:lmisketch} is infeasible.
\end{proof}
\Fref{thm:conversefeasible} suggests a way of sketching feasibility problems of the form
\begin{equation}\label{eq:feasibility}
\sum\limits_{i=1}^mc_iB_i-A\geq0, \qquad c\in\setR^{m}_+.   
\end{equation}
To obtain more concrete bounds on the probability that the original problem is infeasible although the sketched problem is feasible, one would need to know
the parameter $\epsilon$, which is not possible in most applications.

\section{Approximating the value of semidefinite programs through sketching}\label{sec:approxvalue}

We will now show how to approximate with high probability the value of a sketchable SDP by first conjugating both the target matrix and the matrices
that describe the constraints with JLTs and subsequently solving a smaller SDP. 
The next theorem shows that in general it is not possible to approximate with high probability the value of a sketchable SDP using linear sketches.

\begin{thm} \label{thm:nofreeSDP}
Let $\Phi:\M_{2D}\to\setR^{d}$ be a random linear map such that for all sketchable SDPs there exists an algorithm which allows us to estimate the value of an SDP 
up to a constant factor $1\leq \tau < \frac{2}{\sqrt{3}}$ given the sketch
$\{\Phi(A),\Phi(B_1),\ldots,\Phi(B_m)\}$ with probability at least
$9/10$. Then $d=\Omega(D^2)$.
\end{thm}
\begin{proof}
It is well-known that the operator norm of a matrix $G \in \mathcal M_D$ can be computed via an SDP. A linear sketch of the constraints of this SDP would thus allow to approximately compute the operator norm with high probability. However,
in \cite[Theorem 6.5]{Woodruff2014} it was  shown that any algorithm that estimates the operator
norm of a matrix from a linear sketch with probability larger than $9/10$ must have sketch dimension $\Omega(D^2)$. 
\end{proof}

The above result remains true even if we restrict to SDPs that have optimal points with small 
Schatten $1$-norm and low rank. However, we will see below that sketching becomes possible if  the matrices that define
the constraints and the target function have a small Schatten $1$-norm. 
\begin{defi}[Sketched SDP]
Let $A, B_1, \ldots, B_m \in \MM_{D}^{\sym}$, $\eta,\gamma_1,\ldots,\gamma_m\in\setR$ and $\epsilon>0$. Let $X^*\in\mathcal{S}_D^+$ be an optimal
point of the sketchable SDP defined through these matrices. Given that $\tr{X^*}\leq\eta$ and 
given a random matrix 
$S\in\M_{d,D}$, we call the optimization problem

\begin{align}
\mathrm{maximize\qquad} & \tr{SA S^TY} \nonumber \\
\mathrm{subject~to\qquad} & \tr{SB_i S^T Y} \leq \gamma_i + \mu \norm{B_i}_1, \qquad i \in [m]\label{eq:sketchedproblem} \\
					& Y \geq 0 \nonumber
\end{align}
with $\mu = 3 \epsilon \eta$ the sketched SDP. 
\end{defi}
The motivation for defining the sketched SDP is given by the following theorem, which follows directly form \Fref{lem:sketchhs}. 
\begin{thm}\label{thm:uppersdpbysketch}
Let $A, B_1, \ldots, B_m \in \MM_{D}^{\sym}$, $\eta,\gamma_1,\ldots,\gamma_m\in\setR$ and $\epsilon>0$. 
Denote by $\alpha$ the value of the sketchable SDP and assume it is attained at an optimal point $X^*$ which satisfies $\tr{X^*}\leq\eta$.
Moreover, let $S\in\M_{d,D}$ be an 
$(\epsilon,\delta,k)$-JLT, with 
\begin{equation*}
 k\geq\rank{X^*}+\rank{A}+\sum\limits_{i=1}^m\rank{B_i}.
\end{equation*}
Let $\alpha_S$ be the value of the sketched SDP
defined by $A$, $B_i$ and $S$. Then
\begin{equation*}
 \alpha_S+3\epsilon\eta\norm{A}_1\geq\alpha
\end{equation*}
with probability at least $1-\delta$.
\end{thm}
Note that Theorem \ref{thm:uppersdpbysketch} does not rule out the possibility that the value of the sketched problem is much larger than that of the sketchable SDP.
To investigate this issue, we introduce the following:

\begin{defi}[Relaxed SDP]\label{def:relaxedsdp}
Let $A, B_1, \ldots, B_m \in \MM_{D}^{\sym}$, $\eta,\gamma_1,\ldots,\gamma_m\in\setR$ and $\epsilon>0$. Given that an optimal point $X^*$ of the sketchable SDP defined 
through these matrices  satisfies $\tr{X^*}\leq\eta$, we call the optimization problem

\begin{align}
\mathrm{maximize\qquad} & \tr{A X} \nonumber \\
\mathrm{subject~to\qquad} & \tr{B_i X} \leq \gamma_i + \tilde\epsilon_i, \quad i \in [m]  \label{eq:relaxedproblem}\\
					& X \geq 0 \nonumber
\end{align}
with $\tilde\epsilon_i = 3 \epsilon \eta \norm{B_i}_1$ the relaxed SDP. 
\end{defi}
We will obtain lower bounds on the value of the sketchable SDP in terms of the value of the sketched SDP through continuity bounds on the relaxed SDP.
%As the continuity bound we use is for SDPs given in equality form, we begin by giving an equivalent formulation of a sketchable SDP with equality constraints. 
The method of using duality to derive perturbation bounds for a convex optimization problem used here is standard and we refer 
to \cite[Section 5.6]{boydconvex} for a similar derivation. 
We denote by $\AA(\tilde \epsilon)$ the feasible set of the relaxed SDP as in Definition \ref{def:relaxedsdp} for some $\tilde\epsilon \in \setR^m_+$. 
With this notation, $\AA(0)$ is the feasible set of the sketchable SDP. Analogously, we denote by $\alpha(\tilde\epsilon)$ and $\alpha(0)$ the optimal value of the relaxed problem and of the sketchable SDP, respectively. Note that the following result is not probabilistic and holds regardless of the sketching matrix $S$ used.
\begin{thm}\label{thm:uppersketch}
We are in the setting of \Fref{defi:sketchablesdp}. Assume that there exists an $X_0>0$ such that all the constraints of the sketchable SDP
are strictly satisfied and that the dual problem is feasible. Then, the value of the sketched SDP $\alpha_S$ is bounded by
\begin{align*}
 \alpha_S\leq \alpha(0) +  C \norm{y^\ast}_1.
\end{align*}
Here $y^\ast$ is an optimal solution to the dual problem and
\begin{equation*}
C = \max\Set{ 3 \epsilon \eta \norm{B_i}_1 | i \in [m]},
\end{equation*}
where $\eta\geq\tr{X^*}$ for an optimal point $X^*$ of the sketchable SDP.
\end{thm}
\begin{proof}
By Slater's condition \cite[Theorem 2.2]{Watrous2009}, strong duality holds and there is a $y^\ast \geq 0$ which achieves the optimal value.
Note that, given a feasible point $Y$ to the sketched SDP, $S^TYS$ is a feasible point for the relaxed problem by the cyclicity of the trace. Thus, the relaxed SDP gives an upper bound for the sketched SDP.
Hence, for any $X \geq 0$,
\begin{align}
\alpha(0) %&= \sum_{j = 1}^{m} y_j^\ast \gamma_j \nonumber \\
& \geq \sum_{j = 1}^{m} y_j^\ast \gamma_j - \tr{\left[\sum_{i = 1}^{m} y^\ast_i B_i - A\right] X}\nonumber \\
& = \tr{A X}- \sum_{i = 1}^{m} y^\ast_i \left[\tr{B_i X} - \gamma_i\right]. \label{eq:quasilagrange}
\end{align}
The first line holds by duality. If we take the supremum over $X \in \AA(\epsilon)$, we obtain
\begin{equation*}
\alpha(\tilde\epsilon) \leq \alpha(0) + \mscalar{\tilde{\epsilon}}{y^\ast}, 
\end{equation*}
from $y^\ast_i \geq 0$. Here, $\tilde{\epsilon}_i = 3 \eta \epsilon \norm{B_i}_1$, $i \in [m]$.
The assertion then follows by an application of H\"older's inequality.
\end{proof}
Combining \Fref{thm:uppersdpbysketch} and \Fref{thm:uppersketch} it is possible to pick
$\epsilon$ small enough to have an arbitrarily small additive error under some structural assumptions on the SDP.
That is, we need bounds on the Schatten $1$-norms both of $A$ and $B_i$ and we need a bound on the Schatten $1$-norm of an optimal solution to the sketchable SDP. Moreover, we need a bound on the $1$-norm of a dual solution as in \cite{brandao2017}. The following proposition provides a generic bound of this kind.

\begin{prop}
Assume that there exists $X_0 \in \AA(0)$ such that $X_0 > 0$ and 
the constraints are strictly satisfied. Then the value of the sketched SDP $\alpha_S$ is bounded by
\begin{align*}
\alpha_S\leq \alpha(0) +  \epsilon C_1 \left(\alpha(0) - \tr{A X_0}\right) / C_2.
\end{align*}
Here,
\begin{align*}
C_1 &= \max\Set{ 3 \eta \norm{B_i}_1 | i \in [m]}, \\
C_2 &= \min \Set{ \left(\gamma_i - \tr{B_i X_0}\right)| i \in [m]},
\end{align*}
where $\eta\geq\tr{X^*}$ for an optimal point $X^*$ of the sketchable SDP.
\end{prop}
\begin{proof}
We need to bound $\norm{y^\ast}_1$ in \Fref{thm:uppersketch}. From \Fref{eq:quasilagrange} and $\left[\tr{B_i X_0} - \gamma_i\right] < 0$, it follows that
\begin{equation*}
\sum_{i = 1}^{m} y^\ast_i \leq (\alpha(0) - \tr{A X_0}) / \min_{i\in [m]}\left[\gamma_i - \tr{B_i X_0} \right].
\end{equation*}
With $y^\ast \geq 0$, the assertion follows from \Fref{thm:uppersketch}.
\end{proof}

In the case that all the $\gamma_i>0$ for a sketchable SDP we may obtain a bound on the value and an approximate solution to it in a much simpler way.
This class includes the so-called semidefinite packing problems~\cite{Iyengar2005}. These are defined as problems in which all $B_i\geq0$, and so also $\gamma_i\geq0$. Note that we may
set all $\gamma_i=1$ w.l.o.g. by dividing $B_i$ by $\gamma_i$. We then obtain:
\begin{thm}\label{thm:lowerpacking}
For a sketchable SDP with $\gamma_i=1$ and $\nu=3\epsilon\eta \max\limits_{i\in[m]}\norm{B_i}_1$, we have that
\begin{equation}\label{equ:lowerpacking}
 \frac{\alpha_S}{1+\nu}\leq \alpha.
\end{equation}
Moreover, denoting by $X_S^*$ an optimal point of the sketched SDP, we have that $\frac{1}{1+\nu}S^TX_S^*S$ is a feasible point 
of the sketchable SDP that attains this lower bound. 
\end{thm}
\begin{proof}
The lower bound in Equation \eqref{equ:lowerpacking} follows immediately from the cyclicity of the trace, as $\frac{1}{1+\nu}S^TX_S^*S$ is a feasible point of the sketchable SDP.
\end{proof}

\section{Complexity and memory gains}
\label{sec:complexity}
In this section, we will discuss how much we gain by considering the sketched SDP instead of the sketchable SDP. 
We focus on the results of \Fref{sec:approxvalue}, but the discussion carries over to the results of \Fref{sec:feasibility}.
Throughout this section we will assume that we are guaranteed that the Schatten $1$-norms both of an optimal solution to our SDP and of the matrices that define the constraints are $\OO(1)$. We will suppose that upper bounds on the Schatten $1$-norm of both an optimal solution and the constraints are given.

To generate the sketched SDP, we need to compute $m + 1$ matrices of the form $S B S^T$, where $B \in \MM_D$. Each of this computations needs $\OO(\max\Set{\mathrm{nnz}(B), Dd} \epsilon^{-1} \log(k \delta^{-1}))$ operations.
In the worst case, when all matrices $\Set{A, B_1, \ldots, B_m}$ are dense and have full rank, this becomes $\OO(m D^2 \log(mD))$ operations to generate the sketched SDP for fixed $\epsilon$ and $\delta$. We obtain from these considerations and \Fref{thm:uppersdpbysketch}:

\begin{prop}
Let $A, B_1, \ldots, B_m \in \MM_D^{\sym}$, $\gamma_1, \ldots, \gamma_m\in\setR$ of a sketchable SDP be given. Furthermore, let $z := \max\Set{\mathrm{nnz}(A),\mathrm{nnz}(B_1),\ldots, \mathrm{nnz}(B_m)}$ and $\mathrm{SDP}(m,d, \zeta)$ be the
complexity of solving a sketchable SDP (up to accuracy $\zeta$) of dimension $d$. Then a number of 
\begin{equation*}
\OO(\max\Set{z, D \epsilon^{-2} \log(k \delta^{-1})} \epsilon^{-1} m\log(k \delta^{-1}) + \mathrm{SDP}(m,\epsilon^{-2}\log(k\delta^{-1}),\zeta))
\end{equation*}
operations is needed to generate and solve the sketched SDP, where $k$ is defined as in \Fref{thm:uppersdpbysketch}.
\end{prop}
Typically, the costs of forming the sketched matrices $SB_iS^T$ dominates the overall complexity. To compare the above result to other methods for solving SDPs, let us fix $\epsilon$, $\delta$ and $\zeta$. Then, the ellipsoid method \cite[Chapter 3]{Groetschel1988} needs  $\OO(\max\Set{m,D^2}D^6)$ operations to solve the sketchable SDP, whereas using interior point methods we need $\OO(\max\Set{m^3, D^2 m^2, m D^\omega} D^{0.5} \log(D))$ operations \cite[Chapter 5]{Klerk2002}. Here, $\omega $ is the exponent of matrix multiplication. Compared to that, forming the sketched problem and then solving it requires $\OO(m D^2 \log(mD))$ operations.

Another advantage is that using our methods, we need store much smaller matrices.
\begin{prop}
Let $A, B_1, \ldots, B_m \in \MM_D^\sym$, $\gamma_1, \ldots, \gamma_m\in\setR$ be a sketchable SDP. Then we need only store $\OO(m\epsilon^{-4}\log(k/\delta)^2)$ entries for the sketched problem, where 
$k$ is defined as in \Fref{thm:uppersdpbysketch}. 
\end{prop}
Numerical experiments with random instances of SDPs and LMIs that satisfy our requirements indicate that our methods 
may decrease the runtime of SDPs by one order of magnitude. Moreover, they allow us to solve problems in dimensions that are larger by
one order of magnitude.
\section*{Acknowledgements}
We would like to thank Ion Nechita for helpful discussions. 
A.B. acknowledges support from the ISAM Graduate Center at Technical University of Munich.
D.S.F. acknowledges support
from the graduate program TopMath of the Elite Network of Bavaria, the TopMath
Graduate Center of TUM Graduate School at Technical University of Munich. D.S.F.
is supported by the Technical University of Munich – Institute for Advanced Study,
funded by the German Excellence Initiative and the European Union Seventh Framework
Programme under grant agreement no. 291763. 

\bibliographystyle{alpha}
\bibliography{lit}

\begin{thebibliography}{YUTC17}

\bibitem[AL16]{lasserre2016handbook}
M.F. Anjos and J.B. Lasserre.
\newblock {\em Handbook on Semidefinite, Conic and Polynomial Optimization}.
\newblock International Series in Operations Research \& Management Science
  Series. Springer, 2016.

\bibitem[BEFB94]{SDPcontrol}
S.~Boyd, L.~{El~{G}haoui}, E.~Feron, and V.~Balakrishnan.
\newblock {\em Linear Matrix Inequalities in System and Control Theory},
  volume~15 of {\em Studies in Applied Mathematics}.
\newblock {SIAM}, 1994.

\bibitem[BS17]{brandao2017}
F.~G. S.~L. Brandao and K.~M. Svore.
\newblock Quantum speed-ups for solving semidefinite programs.
\newblock In {\em FOCS}, pages 415--426, 2017.

\bibitem[Bub15]{SDPcomplexity}
S.~Bubeck.
\newblock Convex optimization: Algorithms and complexity.
\newblock {\em Found. Trends Mach. Learn.}, 8(3-4):231--357, 2015.

\bibitem[BV04]{boydconvex}
S.~Boyd and L.~Vandenberghe.
\newblock {\em Convex Optimization}.
\newblock Cambridge University Press, 2004.

\bibitem[dK02]{Klerk2002}
E.~de~Klerk.
\newblock {\em Aspects of Semidefinite Programming: Interior Point Algorithms
  and Selected Applications}.
\newblock Applied Optimization. Springer US, 2002.

\bibitem[GLS88]{Groetschel1988}
M.~Grötschel, L.~Lovász, and A.~Schrijver.
\newblock {\em Geometric algorithms and combinatorial optimization}.
\newblock Springer, 1988.

\bibitem[HMS11]{harrowdimred}
A.~W. Harrow, A.~Montanaro, and A.~J. Short.
\newblock Limitations on quantum dimensionality reduction.
\newblock In {\em ICALP}, 2011.

\bibitem[IPS05]{Iyengar2005}
G.~Iyengar, D.~J. Phillips, and C.~Stein.
\newblock Approximation algorithms for semidefinite packing problems with
  applications to maxcut and graph coloring.
\newblock In {\em IPCO}, pages 152--166. Springer, 2005.

\bibitem[KN14]{sparsejltnelson}
D.~M. Kane and J.~Nelson.
\newblock Sparser {J}ohnson-{L}indenstrauss transforms.
\newblock {\em J. ACM}, 61(1):4:1--4:23, 2014.

\bibitem[Ren14]{Renegar2014}
James Renegar.
\newblock Efficient first-order methods for linear programming and semidefinite
  programming.
\newblock {\em ArXiv 1409.5832}, 2014.

\bibitem[SH16]{Stark2016}
C.~J. Stark and A.~W. Harrow.
\newblock Compressibility of positive semidefinite factorizations and quantum
  models.
\newblock {\em IEEE Trans. Inform. Theory}, 62(5):2867--2880, 2016.

\bibitem[VPL15]{2015arXiv150700990V}
K.~{Vu}, P.-L. {Poirion}, and L.~{Liberti}.
\newblock {Using the Johnson-Lindenstrauss lemma in linear and integer
  programming}.
\newblock {\em ArXiv 1507.00990}, 2015.

\bibitem[WA02]{SDPdiscopt}
H.~Wolkowicz and M.~F. Anjos.
\newblock Semidefinite programming for discrete optimization and matrix
  completion problems.
\newblock {\em Discrete Appl. Math.}, 123(1--3):513--577, 2002.

\bibitem[Wat09]{Watrous2009}
J.~Watrous.
\newblock Semidefinite programs for completely bounded norms.
\newblock {\em Theory Comput.}, 5:217--238, 2009.

\bibitem[Wol12]{Wolf2012}
M.~M. Wolf.
\newblock Quantum channels and operations: {G}uided tour.
\newblock Lecture notes available at
  \url{http://www-m5.ma.tum.de/foswiki/pub/M5/Allgemeines/MichaelWolf/QChannelLecture.pdf},
  2012.

\bibitem[Woo14]{Woodruff2014}
D.~P. Woodruff.
\newblock Sketching as a tool for numerical linear algebra.
\newblock {\em Found. Trends Theor. Comput. Sci}, 10(1--2):1--157, 2014.

\bibitem[YUTC17]{sketchydecisions}
A.~{Yurtsever}, M.~{Udell}, J.~A. {Tropp}, and V.~{Cevher}.
\newblock {Sketchy Decisions: Convex Low-Rank Matrix Optimization with Optimal
  Storage}.
\newblock {\em ArXiv 1702.06838}, 2017.

\end{thebibliography}

\end{document}